\documentclass[a4paper,11pt]{amsart}
\usepackage{amsfonts,amssymb,amsthm,cite,amsmath,amstext}
\usepackage{color}
\usepackage[colorlinks,linkcolor=black,citecolor=black]{hyperref}
\usepackage{comment}
\usepackage{framed}

\definecolor{shadecolor}{gray}{0.875}

\newtheorem{thrm}{Theorem}[section]
\newtheorem{lem}[thrm]{Lemma}
\newtheorem{cor}[thrm]{Corollary}
\newtheorem{prop}[thrm]{Proposition}
\newtheorem{conj}[thrm]{Conjecture}

\theoremstyle{definition}

\newtheorem{rmk}[thrm]{Remark}
\newtheorem{ques}[thrm]{Question}

\DeclareMathOperator{\Eff}{\overline{Eff}}

\DeclareMathOperator{\Mov}{Mov}

\DeclareMathOperator{\vol}{vol}

\DeclareMathOperator{\codim}{codim}

\DeclareMathOperator{\PSH}{PSH}

\title{Positivity in convergence of the inverse $\sigma_{n-1}$-flow}
\author{Jian Xiao}
\date{}

\begin{document}

\begin{abstract}
We study positivity in the conjecture proposed by Lejmi and Sz\'{e}kelyhidi on finding effective necessary and sufficient
conditions for solvability of the inverse $\sigma_k$ equation, or equivalently, for convergence of the inverse $\sigma_k$-flow. In particular, for the inverse $\sigma_{n-1}$-flow we partially verify their conjecture by obtaining the desired positivity for $(n-1, n-1)$ cohomology classes. As an application, we also partially verify their conjecture for 3-folds.
\end{abstract}

\maketitle

\tableofcontents

\section{Introduction}
By relating the existence of special K\"ahler metrics with algebro-geometric stability conditions, such as the Yau-Tian-Donaldson conjecture on the existence of constant scalar curvature K\"ahler metric, Lejmi and Sz\'{e}kelyhidi \cite{lejmisze15Jflow} proposed a conjectural numerical criterion for solvability of the inverse $\sigma_k$ equation, or equivalently, for convergence of the inverse $\sigma_k$-flow. Inspired by \cite{collins2014convergence}, Lejmi-Sz\'{e}kelyhidi's conjecture (see \cite[Conjecture 18]{lejmisze15Jflow}) can be somewhat generalized by modifying the numerical condition on $X$ a little bit.

\begin{conj}
\label{conj LS15 sigma k flow}
Let $X$ be a compact K\"ahler manifold of dimension $n$, and fix a positive integer $k$ satisfying $1\leq k<n$.
Let $\omega,\alpha$ be two K\"ahler metrics over $X$ satisfying
\begin{align}
\label{eq over X}
\int_X \omega^n-\frac{n!}{k!(n-k)!}\omega^{n-k}\wedge \alpha^k \geq 0.
\end{align}
Then there exists a K\"ahler metric $\omega' \in \{\omega\}$
such that
\begin{align}
\label{eq n-1 positive}
{\omega'} ^{n-1}- \frac{(n-1)!}{k!(n-k-1)!}{\omega'} ^{n-k-1}\wedge \alpha^k >0
\end{align}
as a smooth $(n-1,n-1)$-form if and only if
\begin{align}
\label{eq numerical positive}
\int_V \omega^p - \frac{p!}{k!(p-k)!} \omega^{p-k}\wedge \alpha^k >0
\end{align}
for every irreducible subvariety of dimension $p$ with
$k\leq p \leq n-1$.
\end{conj}

\subsection{Some backgrounds}
Recall the definition of the inverse $\sigma_k$ equation with respect to two K\"ahler metrics $\omega, \alpha$: we want to find a K\"ahler metric $\omega_\varphi:=\omega + i\partial\bar \partial \varphi \in \{\omega\}$ such that
\begin{equation}\label{eq sigmak}
  \omega_\varphi^n = \frac{n!}{k!(n-k)!}\omega_\varphi^{n-k}\wedge \alpha^k.
\end{equation}
If the above equation can be solved, then $\omega, \alpha$ must satisfy the following numerical equality:
\begin{align}
\int_X \omega^n-\frac{n!}{k!(n-k)!}\omega^{n-k}\wedge \alpha^k = 0.
\end{align}
It has been already noted in \cite{lejmisze15Jflow} that the pointwise positivity of (\ref{eq n-1 positive}), or the solvability of the inverse $\sigma_k$ equation (\ref{eq sigmak}), implies the numerical condition (\ref{eq numerical positive}). Moreover, it is proved in \cite{flmsigmakeq} that the solvability of the inverse $\sigma_k$ equation (\ref{eq sigmak}) is equivalent to the existence of a K\"ahler metric $\omega' \in \{\omega\}$ satisfying (\ref{eq n-1 positive}) -- analogous to the result first proved in  \cite{songweinkJflow} for $k=1$.

More interestingly, by studying a modification of K-stability -- $J$-stability (or more general stability conditions for the inverse $\sigma_k$ equations), and by considering a special class of test configurations arising from deformation to the normal cone of a subvariety, Lejmi and Sz\'{e}kelyhidi got the numerical condition (\ref{eq numerical positive}). Actually, the the numerical condition (\ref{eq numerical positive}) corresponds to the positivity of a modification of Donaldson-Futaki invariant. Thus, similar to the Yau-Tian-Donaldson conjecture, it is natural to ask the statement in Conjecture \ref{conj LS15 sigma k flow}.

For the applications of the inverse $\sigma_k$ equation in K\"ahler geometry (in particular, in the problem on the existence of constant scalar curvature K\"ahler metrics), we refer the readers to \cite{lejmisze15Jflow, collins2014convergence} and the references therein.

\subsection{Previous results} The key (and difficult) part in Conjecture \ref{conj LS15 sigma k flow} is to get the pointwise positivity from the global numerical positivity conditions, this is analogous to the result of Demailly-P\u{a}un \cite{DP04} giving a numerical characterization of the K\"ahler cone.

By studying non linear PDEs, besides other results, the paper \cite{collins2014convergence} confirmed the conjecture for $k=1$ when $X$ is a toric manifold. 

\begin{rmk}
In the paper \cite{collins2014convergence}, the authors studied the following more general equation
\begin{equation*}
  \omega_\varphi^n + c \alpha^n= \frac{n!}{k!(n-k)!}\omega_\varphi^{n-k}\wedge \alpha^k,
\end{equation*}
where $c$ is a topological constant. The advantage of this more general equation is that the hypotheses in Conjecture \ref{conj LS15 sigma k flow} are amenable
to an inductive argument.
\end{rmk}

For arbitrary compact K\"ahler manifolds, in our previous work \cite{xiao16sigmarmk} we mainly obtained the following two results:
\begin{enumerate}
  \item In the case $k=1$ (or the inverse $\sigma_1$ equation), the class $\{\omega - \alpha\}$ must be a K\"ahler class;
  \item In the case $k=n-1$ (or the inverse $\sigma_{n-1}$ equation), the class $\{\omega^{n-1} - \alpha^{n-1}\}$ must be in the dual of the pseudo-effective cone $\Eff^1 (X)$.
\end{enumerate}
In that paper, we also discussed the positivity of $\{\omega^{k} - \alpha^{k}\}$ for general $k$. However, due to the lack of understanding for the singularities of positive $(k,k)$ currents, it seems not easy to prove similar positivity for the class $\{\omega^{k} - \alpha^{k}\}$ (see \cite[Question 3.1]{xiao16sigmarmk}).

\subsection{Main results} 
For the most general situation of Conjecture \ref{conj LS15 sigma k flow} in higher dimensional case, as pointed out in \cite{lejmisze15Jflow}, it may be necessary to refine the conjecture allowing for more general test-configurations.
On the other hand, Conjecture \ref{conj LS15 sigma k flow} would imply the following weaker conjecture:
\begin{conj}[\textbf{weak conjecture}]\label{weakconj}
Under the assumptions of Conjecture \ref{conj LS15 sigma k flow}, the class $$\{{\omega} ^{n-1}- \frac{(n-1)!}{k!(n-k-1)!}{\omega} ^{n-k-1}\wedge \alpha^k\}$$ contains a smooth strictly positive $(n-1, n-1)$ form.
\end{conj}

In this paper, we mainly focus on the inverse $\sigma_{n-1}$ equation, and verify the weak Conjecture \ref{weakconj} in this case (see Theorem \ref{main result}). As an immediate consequence, applying \cite{xiao16sigmarmk} yields a solution to Conjecture \ref{weakconj} for K\"ahler 3-folds (see Corollary \ref{cor}).

In the algebraic geometry setting, one often needs to consider movable curve classes rather than complete intersection classes. By using the refined structure of movable cone of curves, we give a sufficient condition such that the difference of two movable curve classes is in the interior of the movable cone (see Theorem \ref{main result1}), which
may be useful in the study of stability conditions of vector bundles with respect to movable classes.

\subsection{Ingredients in the proofs}
The proofs of Conjecture \ref{weakconj} for $k=n-1$ and its extension to movable $(n-1, n-1)$ classes depend on the following ingredients:
\begin{itemize}
  \item Divisorial Zariski decomposition for pseudo-effective $(1,1)$ classes, which is given by \cite{Bou04} (see also \cite{Nak04} in the algebraic setting);
  \item Morse type bigness criterion for movable $(n-1, n-1)$ classes, which is noted in
      \cite{xiao2014movable} or
      \cite[Section 4]{lehmann2016convexity}
      (see also \cite[Remark 3.1]{xiao13weakmorse});
  \item Restricted version of ``reverse Khovanskii-Teissier inequalities''\footnote
      {See \cite[Remark 4.18]{lehmann2016convexity} for this terminology.},
      which follows from
      \cite{popovici2016sufficientbig, popovici15} (or \cite{nystrom2016duality} for projective manifolds).
  \item Some properties of positive products, which is proved in \cite[Lemma 6.21]{fl14}.
  \item The refined structure of the movable cone $\Mov_1 (X)$, which follows from \cite{lehmann2016positiivty}.
\end{itemize}
Furthermore, in order to obtain the stronger pointwise positivity when $X$ is a projective manifold, we also need
\begin{itemize}
  \item The duality of (transcendental) cones $\Eff^1 (X)^* =\Mov_1 (X)$, which is proved in \cite{nystrom2016duality}.
\end{itemize}
\subsubsection{Sketch of the proofs} We give the sketch for the proof of Conjecture \ref{weakconj} for $k=n-1$ (some steps in its extension to movable classes are similar). By Morse type bigness criterion for movable $(n-1, n-1)$ classes, over every birational model we conclude that there exists a positive $(n-1, n-1)$ current in the pull-back class of $\{\omega^{n-1}-\alpha^{n-1}\}$.
Applying divisorial Zariski decomposition for $(1,1)$ classes and the numerical condition for prime divisors, we prove that $\{\omega^{n-1}-\alpha^{n-1}\}$ must be a movable $(n-1, n-1)$ class (at least when $X$ is projective), or equivalently, $\{\omega^{n-1}-\alpha^{n-1}\} \in \Eff^1 (X)^*$.
By some kind of restricted version of ``reverse Khovanskii-Teissier inequalities'', we improve the positivity and prove that $\{\omega^{n-1}-\alpha^{n-1}\}$ must be an interior point of $\Eff^1 (X)^*$.
At last, we apply the duality of cones to obtain the desired pointwise strict positivity, which in turn follows from the geometric form of Hahn-Banach theorem
(see e.g \cite{lamarilemma, tomamobcurve}).

\subsection{Organization of the paper}
In Section \ref{secprelim}, we briefly introduce some concepts for positivity and present the key ingredients that are needed in the proof. Section \ref{secproof} devotes to the proof of our main results. In Section \ref{secdisc}, we present some discussions on a general version of restricted reverse Khovanskii-Teissier inequalities.

\subsection*{Acknowledgement}
We would like to thank B.~Lehmann for his support of this work, and thank T.~Collins, G.~Sz\'{e}kelyhidi, V.~Tosatti and B.~Weinkove for their interests and correspondences. In particular, we would like to thank P.~Eyssidieux and M.~Lejmi for drawing our attention to this interesting problem when M.~Lejmi gave a seminar talk at Institut Fourier in March of 2015.

\section{Preliminaries}\label{secprelim}
\subsection{Positivity}
Let $X$ be a compact K\"ahler manifold of dimension $n$. We will let $H^{1,1}(X, \mathbb{R})$ denote the real Bott-Chern cohomology group of bidegree $(1,1)$. A Bott-Chern $(1,1)$ class is called nef if it lies on the closure of the K\"ahler cone; and it is called pseudo-effective if it contains a $d$-closed positive current.  A pseudo-effective $(1,1)$ class $\beta$ is called movable if for any irreducible divisor $Y$ the Lelong number  $\nu(\beta,y)$ (or minimal multiplicity as in \cite{Bou04}) vanishes at a very general point $y \in Y$, or equivalently, $\beta$ is called movable if for any $\varepsilon>0$, there exists a modification $\mu: \widehat{X} \rightarrow X$ and a K\"ahler class $\widehat{\omega}$ on $\widehat{X}$ such that
\begin{equation*}
  \mu_*\widehat{\omega} = \beta + \varepsilon \omega,
\end{equation*}
where $\omega$ is any fixed K\"ahler class. We will be interested in the following cones in $H^{1,1}(X,\mathbb{R})$:
\begin{itemize}
\item $\Eff^{1}(X)$: the cone of pseudo-effective $(1,1)$-classes.
\item $\Mov^{1}(X)$: the cone of movable $(1,1)$-classes.
\end{itemize}
The interior point of $\Eff^{1}(X)$ is called big class. Let $H^{n-1,n-1}(X, \mathbb{R})$ denote the real Bott-Chern cohomology group of bidegree $(n-1,n-1)$. We will be interested in the following cones in $H^{n-1,n-1}(X,\mathbb{R})$:
\begin{itemize}
\item $\Eff_1 (X)$: the cone of pseudo-effective $(n-1,n-1)$-classes;
\item $\Mov_{1}(X)$: the cone of movable $(n-1,n-1)$-classes.
\end{itemize}
Recall that an $(n-1,n-1)$-class is called pseudo-effective if it contains a positive $(n-1, n-1)$ current, and $\Mov_1 (X)$ is the closed cone generated by classes of the form $\mu_* (\widetilde{A}_1 \cdot ... \cdot \widetilde{A}_{n-1})$, where $\mu$ is a modification and the $\widetilde{A}_i$ are K\"ahler classes upstairs. An irreducible curve $C$ on a projective variety is called movable if it is a member of an algebraic family that covers the variety.

\subsubsection{Positive products}
Let $X$ be a compact K\"ahler manifold of dimension $n$. Assume that $L_1, ..., L_r$ are big $(1,1)$ classes, that is, every class $L_i$ contains a K\"ahler current. By the theory developed in \cite{BEGZ10MAbig} (see also \cite{BDPP13}, \cite{BFJ09}), we can associate to $L_1, ..., L_r$ a positive class in $H^{r,r} (X, \mathbb{R})$, denoted by $\langle L_1 \cdot ...\cdot L_r\rangle$. It is defined as the class of the non-pluripolar product of positive current with minimal singularities, that is,
\begin{equation*}
  \langle L_1 \cdot ...\cdot L_r\rangle := \{\langle T_{1, \min} \wedge ...\wedge T_{r, \min}\rangle\}
\end{equation*}
where $\langle T_{1, \min} \wedge ...\wedge T_{r, \min}\rangle$ is the non-pluripolar product. Note that such a current $T_{i, \min} \in L_i$ always exists: let $\theta \in L_i$ be a smooth $(1,1)$ form and let
$$V_{\theta} := \sup \{\varphi \in \PSH(X, \theta)|\ \varphi \leq 0\},$$
then $\theta+ dd^c V_{\theta}$ is a positive current with minimal singularities. There may be many positive currents with minimal singularities in a class, but it is proved in \cite{BEGZ10MAbig} that the positive product $\langle L_1 \cdot ...\cdot L_r\rangle$ does not depend on the choices. For non big pseudo-effective classes, their positive products are defined by taking limits for big ones.
Moreover, if the $L_i$ are nef, then $\langle L_1 \cdot ...\cdot L_r\rangle=L_1 \cdot ...\cdot L_r$ is the usual intersection.

For positive products, applying the result of \cite{BFJ09} (and \cite{nystrom2016duality} for the transcendental situation), the following result was proved in \cite{fl14}.

\begin{lem}\label{pullback}
Let $X$ be a projective manifold of dimension $n$, and let $\alpha\in H^{1,1} (X, \mathbb{R})$ be a pseudo-effective class. Then for any modification $\mu:\widehat{X}\rightarrow X$ we have
\begin{equation*}
  \mu^*\langle \alpha^{n-1}\rangle
  = \langle(\mu^* \alpha)^{n-1}\rangle.
\end{equation*}
\end{lem}

For the structure of $\Mov_1 (X)$, by \cite{lehmann2016positiivty} we have:

\begin{lem}\label{movablestruct}
Let $X$ be a projective manifold of dimension $n$, and let $\gamma$ be a movable $(n-1, n-1)$ class. If $\gamma\cdot \beta>0$ for every non zero movable $(1,1)$ class $\beta$, then $\gamma=\langle L^{n-1}\rangle$ for a big class $L$. Furthermore, $\gamma$ is an interior point of $\Mov_1 (X)$ if and only if $\gamma=\langle L^{n-1}\rangle$ for a big class with $\codim \mathbb{B}_+ (L) \geq 2$, and in this case it has strictly positive intersection against any non zero movable $(1,1)$ class.
\end{lem}

\subsection{Divisorial Zariski decomposition} By the main result of \cite{Bou04}, we have the divisorial Zariski decomposition for pseudo-effective $(1,1)$ classes on any compact complex manifold.

\begin{lem}\label{zariskidecom}
Let $X$ be a compact complex manifold, and let $\alpha \in \Eff^1 (X)$ be a pseudo-effective $(1,1)$ class, then $\alpha$ admits a decomposition $\alpha = P(\alpha) + N(\alpha)$
such that $P(\alpha)$ is movable, and $N(\alpha)$ is an effective divisor class which contains only one positive current.
\end{lem}

\subsection{Morse type bigness criterion for $(n-1, n-1)$ classes}
By using basic properties of positive products of pseudo-effective $(1,1)$ classes, in \cite{xiao2014movable} we observed that the main result of \cite{popovici2016sufficientbig} can be generalized from nef classes to pseudo-effective classes. In this way, we proved the following Morse type bigness for the difference of two movable $(n-1, n-1)$ classes (see \cite[Theorem 1.3]{xiao2014movable}).

\begin{lem}\label{morsecurve}
Let $X$ be a compact K\"ahler manifold of dimension $n$, and let $\alpha, \beta\in \Eff^1 (X)$ be two pseudo-effective classes. Then $\vol(\alpha)-n\alpha\cdot \langle \beta^{n-1}\rangle>0$ implies that there exists a strictly positive $(n-1, n-1)$ current in the Bott-Chern class $\langle \alpha^{n-1}\rangle-\langle \beta^{n-1}\rangle$, or equivalently, it is an interior point of $\Eff_1 (X)$. In particular, we have a Morse type bigness criterion for the difference of two complete intersection classes.
\end{lem}

\begin{rmk}
Indeed, in \cite[Section 4]{lehmann2016convexity} we studied Morse type inequality with respect to a subcone. The above result can be restated as: $\Eff_1(X)$ (with a suitable volume type function) satisfies a Morse type inequality with respect to its subcone $\Mov_1 (X)$. And from the viewpoint of duality in convex analysis, it is proved in \cite[Section 4]{lehmann2016convexity} that the polar transform gives a natural way of translating cone positivity conditions from a cone to its dual cone, and the above result can also be derived from (and fits very well with) the abstract setting.
\end{rmk}

\subsection{Reverse Khovanskii-Teissier inequality} By the result of \cite{popovici2016sufficientbig} (or \cite{nystrom2016duality} for projective manifolds) and its generalization to pseudo-effective classes, we have the following result (see e.g. \cite[Section 3.4]{xiao2014movable}).

\begin{lem}\label{reverseKT}
Let $X$ be a compact K\"ahler manifold of dimension $n$, and let $\alpha, \beta\in \Eff^1 (X)$ be two pseudo-effective classes. Then for any nef class $N$ we have
\begin{equation*}
  n(N\cdot\langle \alpha^{n-1}\rangle)(\alpha\cdot \langle \beta^{n-1} \rangle) \geq \vol(\alpha) (N\cdot \langle \beta^{n-1} \rangle).
\end{equation*}
\end{lem}

\begin{rmk}
It is noted in \cite[Section 4.2]{lehmann2016convexity} that, once we have a More type inequality, then we can also have some kind of ``reverse'' Khovanskii-Teissier inequality (and it is useful when we translate the positivity in a cone to its dual cone).
\end{rmk}

\subsubsection{Restricted version} The above result is sufficient in the proof of Theorem \ref{main result1}. For a projective manifold, because of the existence of ample divisors, it is also sufficient to improve the positivity in the proof of Theorem \ref{main result}.
But for non-projective K\"ahler manifold, we need the following result which follows from
\cite[inequality (63)]{popovici15}.
Indeed, the paper \cite{popovici15} gives more general and stronger results. For our application, we need one of its corollaries.

\begin{lem}\label{thm resKT}
Let $X$ be a compact K\"ahler manifold of dimension $n$, and let $\omega, \alpha, \beta$ be three K\"ahler classes. Then we have
\begin{equation*}
  (n-1) (\omega^2 \cdot \alpha^{n-2}) (\omega^{n-1}\cdot \beta) \geq \omega^n (\omega \cdot \alpha^{n-2} \cdot \beta).
\end{equation*}
\end{lem}

\begin{rmk}
We use the terminology ``restricted'' because if $\omega, \alpha, \beta$ are K\"ahler classes and $H$ is a prime divisor, then Lemma \ref{reverseKT} implies:
\begin{equation*}
(n-1) (\omega_H \cdot \alpha_H^{n-2}) (\omega_H ^{n-2}\cdot \beta_H) \geq \omega_H ^{n-1} (\alpha_H ^{n-2} \cdot \beta_H),
\end{equation*}
where $\omega_H, \alpha_H, \beta_H$ are the restrictions of the classes on $H$.
In particular, if $\omega$ is ample, then Lemma \ref{thm resKT} follows from Lemma \ref{reverseKT}; see Section \ref{secdisc} for more related discussions.
\end{rmk}

\subsection{The duality of cones and pointwise positivity}
By confirming a conjecture of \cite{BDPP13}, for projective manifolds, the following duality of (transcendental) cones is proved in \cite{nystrom2016duality}:
\begin{equation*}
  \Eff^1 (X)^* =\Mov_1 (X).
\end{equation*}

By using the geometric form of Hanhn-Banach theorem, by \cite{lamarilemma} and \cite{tomamobcurve} (see also \cite[Appendix]{fuxiao14kcone} for the extension of Toma's result to K\"ahler setting), we have

\begin{lem}\label{thm conedual pointwisepos}
Let $X$ be a compact K\"ahler manifold of dimension $n$, then we have:
\begin{itemize}
  \item Every interior point of $\Eff^1 (X)^*$ can be represented by a smooth strictly positive $(n-1, n-1)$ form (up to some form $\partial \psi + \overline{\partial \psi}$);
  \item Furthermore,
  if the cone duality $\Eff^1 (X)^* =\Mov_1 (X)$ holds, then every interior point of $\Eff^1 (X)^*$ can be represented by a smooth $d$-closed strictly positive $(n-1, n-1)$ form (up to some form $i\partial \bar \partial \theta$).
\end{itemize}
\end{lem}

\section{The main results}\label{secproof}
\subsection{The inverse $\sigma_{n-1}$ equation} We first prove the following result, verifying Conjecture \ref{weakconj} for the inverse $\sigma_{n-1}$ equation.

\begin{thrm}\label{main result}
Let $X$ be a compact K\"ahler manifold of dimension $n$, and let $\omega, \alpha$ be two K\"ahler metrics such that
\begin{align}\label{eq 1}
 \int_X \omega^n-n\omega\wedge \alpha^{n-1} \geq 0,
\end{align}
and
\begin{align}\label{eq 2}
  \int_{E} \omega^{n-1} -\alpha^{n-1} >0
\end{align}
for every prime divisor $E$. Then there exists a smooth $(n-2, n-1)$ form $\psi$ such that
\begin{equation*}
  \omega^{n-1} -\alpha^{n-1} + \partial \psi + \overline{\partial \psi}>0
\end{equation*}
as a smooth $(n-1,n-1)$-form. Furthermore, if $X$ is a projective manifold, then there exists a smooth $(n-2, n-2)$ form $\theta$ such that
\begin{equation*}
  \omega^{n-1} -\alpha^{n-1} + i\partial \bar \partial \theta >0
\end{equation*}
as a smooth $(n-1,n-1)$-form.
\end{thrm}

\begin{rmk}
In some sense, this result can be seen as the $(n-1, n-1)$ class version of Demailly-P\u{a}un's theorem. The positivity we get in Theorem \ref{main result} means that, instead of finding the conjectural K\"ahler metric $\omega'$ in the K\"ahler class $\{\omega\}$, we get a special (Gauduchon or balanced) Hermitian metric $\widetilde{\omega}$ such that $\widetilde{\omega}^{n-1}$ is in the class $\{\omega^{n-1}\}$ and satisfy $\widetilde{\omega}^{n-1}-\alpha^{n-1}>0$ as a smooth $(n-1, n-1)$ form.
\end{rmk}

\begin{proof}[Proof of Theorem \ref{main result}]
To simplify the notations, we denote the corresponding K\"ahler class by the same symbol as the K\"ahler metric.
It has been already noted in \cite{xiao16sigmarmk} that, under the assumptions in Theorem \ref{main result}, the class $\omega^{n-1} - \alpha^{n-1}$ must be a point of $\Eff^1 (X)^*$. For completeness, in the following we will recall some arguments.\\

We first consider the case when $\omega^n - n\omega\cdot \alpha^{n-1}>0$. In order to prove that $\omega^{n-1} - \alpha^{n-1}$ is an interior point of $\Eff^1 (X)^*$, we need to verify the following statement: the inequality
\begin{equation}\label{eq interior}
(\omega^n - \alpha^{n-1})\cdot \beta >0
\end{equation}
holds for any non-zero pseudo-effective class $\beta\in \Eff^1 (X)$.
By divisorial Zariski decomposition (Lemma \ref{zariskidecom}), $\beta$ can be decomposed as $$\beta = P(\beta)+N(\beta),$$ where $P(\beta)$ is movable and $N(\beta)$ is an effective divisor class.  Firstly, note that by the numerical condition for prime divisors, we always have
\begin{equation*}
  (\omega^n - \alpha^{n-1})\cdot N(\beta)\geq 0,
\end{equation*}
and the above inequality is strict if $N(\beta)\neq 0$. On the other hand, we claim:
\begin{equation*}
  (\omega^n - \alpha^{n-1})\cdot P(\beta)\geq 0,
\end{equation*}
and the inequality is strict whenever $P(\beta)\neq 0$. Then it is clear that (\ref{eq interior}) follows from these two statements.

Now we prove our claim. Assume $P(\beta)\neq 0$. Since $P(\beta)$ is movable, for any $\varepsilon>0$ there exists a modification $\mu: \widehat{X} \rightarrow X$ and a K\"ahler class $\widehat{\omega}$ on $\widehat{X}$ such that
\begin{equation*}
  \mu_* \widehat{\omega} = P(\beta)+\varepsilon \omega.
\end{equation*}
We estimate the intersection number $(\omega^{n-1} - \alpha^{n-1})\cdot \mu_* \widehat{\omega}$.  By the numerical condition on $X$, we have $$\mu^* \omega^n -n \mu^*\omega \cdot (\mu^* \alpha)^{n-1}>0.$$
Applying Lemma \ref{morsecurve} yields a strictly positive $(n-1, n-1)$ current in the class $(\mu^* \omega)^{n-1}-(\mu^* \alpha)^{n-1}$. This implies
\begin{align*}
  (\omega^{n-1} - \alpha^{n-1})\cdot (P(\beta)+\varepsilon \omega)&=(\omega^{n-1} - \alpha^{n-1})\cdot \mu_* \widehat{\omega}\\
  &=(\mu^* \omega^{n-1} - \mu^* \alpha^{n-1})\cdot \widehat{\omega}\\
  &>0.
\end{align*}
Let $\varepsilon \downarrow 0$, we conclude that $(\omega^{n-1} - \alpha^{n-1})\cdot P(\beta) \geq 0$. Note that, in the proof of this inequality, we only use the numerical condition on $X$. Since we have assumed $\omega^n -n \omega\cdot \alpha^{n-1}>0$, for $\delta>0$ small enough, we also have
\begin{equation*}
  \omega^n -n \omega\cdot (\alpha+\delta \omega)^{n-1}>0.
\end{equation*}
Applying the same argument to the class $\omega^{n-1}-(\alpha+\delta \omega)^{n-1}$, we get
\begin{equation*}
 (\omega^{n-1}-(\alpha+\delta \omega)^{n-1})\cdot P(\beta)\geq 0,
\end{equation*}
which implies $(\omega^{n-1} - \alpha^{n-1})\cdot P(\beta) > 0$ (since $P(\beta)\neq 0$). \\

Next we consider the case when $\omega^n - n\omega\cdot \alpha^{n-1}=0$. Then for $\delta>0$ small enough, the numerical inequalities are strict for the classes $\omega$ and $(1-\delta)\alpha$. Using the same argument as in the first case and letting $\delta$ tend to $0$, we also get that
\begin{equation*}
  (\omega^{n-1} - \alpha^{n-1})\cdot \beta \geq 0
\end{equation*}
for any pseudo-effective class $\beta$. To conclude that $\omega^{n-1} - \alpha^{n-1}$ is an interior point, we also need to verify
\begin{equation*}
  (\omega^{n-1} - \alpha^{n-1})\cdot P(\beta) > 0
\end{equation*}
whenever $P(\beta)\neq 0$.

To this end, our strategy is to find a K\"ahler class $H$, and three strictly positive constants $t_1, t_2, \varepsilon$ ($\varepsilon$ will be sufficiently small) such that
\begin{itemize}
  \item $(\omega - \varepsilon t_1 H)^n -n (\omega - \varepsilon t_1 H)\cdot (\alpha - \varepsilon t_2 H)^{n-1} >0$,
  \item $\left((\omega - \varepsilon t_1 H)^{n-1} - (\alpha - \varepsilon t_2 H)^{n-1}\right)\cdot P(\beta)< (\omega^{n-1}-\alpha^{n-1})\cdot P(\beta)$.
\end{itemize}
By the proof for the first case, the first inequality implies that
\begin{equation*}
  \left((\omega - \varepsilon t_1 H)^{n-1} - (\alpha - \varepsilon t_2 H)^{n-1}\right)\cdot P(\beta) \geq 0.
\end{equation*}
Thus, if we also have the second inequality, then we get the desired result.

We claim that there exist a K\"ahler class $H$, three strictly positive constants $t_1, t_2, \varepsilon$ satisfying both inequalities. Note that, for $\varepsilon>0$ sufficiently small we have
\begin{align*}
  &(\omega - \varepsilon t_1 H)^n -n (\omega - \varepsilon t_1 H)\cdot (\alpha - \varepsilon t_2 H)^{n-1}\\
  &=\omega^n - n \varepsilon t_1 H \cdot \omega^{n-1} - n \omega\cdot \alpha^{n-1} + n(n-1)\varepsilon t_2 H\cdot \omega \cdot \alpha^{n-2} + n\varepsilon t_1 H\cdot \alpha^{n-1} + O(\varepsilon^2)\\
  &=n\varepsilon \left((n-1)t_2 H\cdot \omega \cdot \alpha^{n-2} - t_1 H \cdot (\omega^{n-1}-\alpha^{n-1}) \right)+ O(\varepsilon^2),
\end{align*}
where the last equality follows since $\omega^n -n\omega\cdot \alpha^{n-1}=0$. Similarly, we also have
\begin{align*}
  &\left((\omega - \varepsilon t_1 H)^{n-1} - (\alpha - \varepsilon t_2 H)^{n-1}\right)\cdot P(\beta)\\
  &=(\omega^{n-1}-\alpha^{n-1})\cdot P(\beta) + (n-1)\varepsilon (t_2 H\cdot \alpha^{n-2}-t_1 H\cdot \omega^{n-2})\cdot P(\beta) + O(\varepsilon^2).
\end{align*}
Let $t=t_1 /t_2$.
In the case when $H\cdot(\omega^{n-1}-\alpha^{n-1})=0$ (which is impossible as it will lead contradiction), it is clear
that we only need to take $t_1, t_2 \in (0,1)$ such that $t$ is sufficiently large.

In the case when $H\cdot(\omega^{n-1}-\alpha^{n-1})>0$, then it is easy to see that the existence of $H, t_1, t_2, \varepsilon$ is equivalent to the following inequality
\begin{align}
  (n-1)(H\cdot \alpha^{n-2}\cdot \omega)(H\cdot \omega^{n-2}\cdot P(\beta))> (H\cdot \alpha^{n-2}\cdot P(\beta))(H \cdot (\omega^{n-1}-\alpha^{n-1})).
\end{align}

Since $H\cdot \alpha^{n-2}\cdot P(\beta)>0$ and $H\cdot \alpha^{n-1}>0$, it is sufficient to prove that
\begin{align}\label{eq restrictedKT}
  (n-1)(H\cdot \alpha^{n-2}\cdot \omega)(H\cdot \omega^{n-2}\cdot P(\beta))\geq (H\cdot \alpha^{n-2}\cdot P(\beta))(H \cdot \omega^{n-1}).
\end{align}

We claim that, in order to prove the inequality (\ref{eq restrictedKT}) for any movable class $P(\beta)$, it is sufficient to establish the following inequality:
\begin{align}\label{eq restrictedKT1}
  (n-1)(H\cdot \alpha^{n-2}\cdot \omega)(H\cdot \omega^{n-2}\cdot N)\geq (H\cdot \alpha^{n-2}\cdot N)(H \cdot \omega^{n-1}).
\end{align}
for any K\"ahler or nef class $N$. This is clear, since by taking limits we can assume $P(\beta)=\nu_* \widehat{\omega}$ for some K\"ahler classes upstairs. \\

Now we prove that there always exists some K\"ahler classes such that (\ref{eq restrictedKT1}) holds.

In the case when $X$ is projective, we can take $H$ to be the class of any irreducible ample divisor. Indeed, if $H$ is an ample divisor, then (\ref{eq restrictedKT1}) is equivalent to
\begin{equation}\label{eq proj}
 (n-1) (\alpha_H ^{n-2}\cdot \omega_H)(\omega_H ^{n-2}\cdot N_H)\geq \omega_H^{n-1} (\alpha_H ^{n-2}\cdot N_H).
\end{equation}
And this just follows from Lemma \ref{reverseKT}, by considering the restricted classes on $H$.

In the case when $X$ is not projective, we expect that (\ref{eq restrictedKT1}) also holds for any K\"ahler classes (see Section \ref{secdisc} for more discussions). In our setting, by Lemma \ref{thm resKT} we observe that it is sufficient to take $H=\omega$.\\

Thus under the assumptions in Theorem \ref{main result}, the class $\omega^{n-1}-\alpha^{n-1}$ must be an interior point of $\Eff^1 (X)^*$. Applying Lemma \ref{thm conedual pointwisepos} implies the existence of a smooth $(n-2, n-1)$ form $\psi$ such that
\begin{equation*}
  \omega^{n-1} -\alpha^{n-1} + \partial \psi + \overline{\partial \psi}>0
\end{equation*}
as a smooth $(n-1,n-1)$-form. And if $X$ is a projective manifold, then Lemma \ref{thm conedual pointwisepos} implies the existence of a smooth $(n-2, n-2)$ form $\theta$ such that
\begin{equation*}
  \omega^{n-1} -\alpha^{n-1} + i\partial \bar \partial \theta >0
\end{equation*}
as a smooth $(n-1,n-1)$-form. This finishes the proof of Theorem \ref{main result}.
\end{proof}

\subsection{K\"ahler 3-folds} As an application, Theorem \ref{main result} and \cite{xiao16sigmarmk} yield a solution to Conjecture \ref{weakconj} for K\"ahler 3-folds. Note that, we only need to verify the case when $k=1$.

\begin{cor}\label{cor}
Let $X$ be a compact K\"ahler manifold of dimension 3. Assume that $\omega, \alpha$ are two K\"ahler metrics satisfying
\begin{equation*}
  \int_X \omega^3 -3 \omega^2\wedge \alpha \geq 0,
\end{equation*}
and
\begin{equation*}
  \int_V \omega^p - p \omega^{p-1}\wedge \alpha >0
\end{equation*}
for every irreducible analytic subset $V$ with $\dim V = p$, $p=1, 2$. Then there exists a smooth $(1, 2)$ form $\psi$ such that
\begin{equation*}
  \omega^{2} -2\omega\wedge \alpha + \partial \psi + \overline{\partial \psi}>0
\end{equation*}
as a smooth $(2,2)$-form. Furthermore, if $X$ is a projective manifold, then there exists a smooth $(1, 1)$ form $\theta$ such that
\begin{equation*}
  \omega^{2} -2\omega\wedge \alpha + i\partial \bar \partial \theta >0
\end{equation*}
as a smooth $(2,2)$-form.
\end{cor}

\begin{proof}
We use the same symbol to denote a $(1,1)$ form and its Bott-Chern class. Note that
\begin{equation*}
  \omega^2 -2\omega\cdot \alpha = (\omega-\alpha)^2-\alpha^2,
\end{equation*}
and
\begin{equation*}
  (\omega-\alpha)^3-3(\omega-\alpha)\cdot \alpha^2 =(\omega^3 -3\omega^2 \cdot \alpha) + 2 \alpha^3.
\end{equation*}
The numerical conditions imply
\begin{equation*}
  \int_X (\omega-\alpha)^3-3(\omega-\alpha)\cdot \alpha^2 >0
\end{equation*}
and
\begin{equation*}
  \int_E (\omega-\alpha)^2-\alpha^2 >0
\end{equation*}
for every prime divisor $E$.

By \cite{xiao16sigmarmk}, the class $\omega-\alpha$ is a K\"ahler class. Applying Theorem \ref{main result} to the K\"ahler classes $\omega-\alpha$ and $\alpha$, we get the desired pointwise positivity.
\end{proof}

\begin{rmk}
It would be interesting to see if the proof of Corollary \ref{cor} can be generalized to higher dimension. Unfortunately, there is certain difficulty even for $n=4$. More precisely, under the assumptions in Conjecture \ref{conj LS15 sigma k flow}, we want to study the positivity of
\begin{equation*}
  \omega^3 - 3 \omega^2 \cdot \alpha = (\omega-\alpha)^3 -(3\omega\cdot \alpha^2 -\alpha^3).
\end{equation*}
By \cite{xiao16sigmarmk}, the class $\omega-\alpha$ is K\"ahler, which implies that $3\omega\cdot \alpha^2 -\alpha^3=\alpha^2 \cdot (3\omega -\alpha)$ is a complete intersection class. At least when $X$ is projective, applying the refined structure of movable cone in \cite{lehmann2016positiivty} (see Lemma \ref{movablestruct}) implies
\begin{equation*}
  3\omega\cdot \alpha^2 -\alpha^3 = \langle L^3 \rangle
\end{equation*}
for a unique big and movable class $L$. By Morse type bigness criterion for movable $(n-1, n-1)$ classes (see Lemma \ref{morsecurve}), if 
\begin{align*}
  &(\omega-\alpha)^4 -4 (\omega-\alpha)\cdot (3\omega\cdot \alpha^2 -\alpha^3)\\ 
  &= (\omega-\alpha)^4 -4 (\omega-\alpha)\cdot \langle L^3 \rangle >0,
\end{align*}
then the class $\omega^3 - 3 \omega^2 \cdot \alpha$ must contain a strictly positive $(3, 3)$-current, and as the following proof for Theorem \ref{main result1}, this will yield a solution to Conjecture \ref{weakconj} in this case. However, note that
\begin{align*}
  &(\omega-\alpha)^4 -4 (\omega-\alpha)\cdot (3\omega\cdot \alpha^2 -\alpha^3)\\
  &= (\omega^4 -4\omega^3 \cdot \alpha)+ (12 \omega\cdot \alpha^3 -6 \omega^2 \cdot \alpha^2 - 3 \alpha^4)
\end{align*}
where the second term $12 \omega\cdot \alpha^3 -6 \omega^2 \cdot \alpha^2 - 3 \alpha^4$ may be strictly negative, thus if $\omega^4 -4\omega^3 \cdot \alpha =0$ then we may get a negative numerical condition on $X$.
\end{rmk}

\subsection{Differences of movable curve classes} In the algebraic geometry setting, we usually need to deal with movable $(n-1,n-1)$ classes rather than complete intersections. Let $X$ be a projective manifold, and let $\gamma_1, \gamma_2$ be two (transcendental) movable $(n-1,n-1)$ classes, we ask whether there is a similar intersection-theoretic criterion, as Theorem \ref{main result}, such that $\gamma_1 - \gamma_2$ is an interior point of $\Mov_1 (X)$. This might be applied to the study of stability conditions of vector bundles with respect to (transcendental) movable classes.

By the refined structure of $\Mov_1 (X)$ (see \cite{lehmann2016positiivty}), every interior point in this cone can be written as the positive product of a unique big movable $(1,1)$ class -- which is called the $(n-1)$th root, it is natural to ask whether the above result can be extended to pseudo-effective classes by using positive products. In this direction, we have:

\begin{thrm}\label{main result1}
Let $X$ be a projective manifold of dimension $n$, and let $\gamma_1, \gamma_2$ be two movable $(n-1, n-1)$ classes. Assume that $\gamma_1 = \langle\omega^{n-1}\rangle, \gamma_2=\langle\alpha^{n-2} \rangle$ for two big classes $\omega, \alpha$ satisfying
\begin{align}\label{eq 3}
 \vol(\omega)-n\omega\cdot \gamma_2 \geq 0
\end{align}
and
\begin{align}\label{eq 4}
  (\gamma_1 -\gamma_2)\cdot E \geq 0
\end{align}
for every prime divisor class $E$. Then we have:
\begin{itemize}
  \item $\gamma_1 -\gamma_2$ must be a movable $(n-1, n-1)$ class;
  \item Furthermore, if we assume that the augmented base locus of $\alpha$, $\mathbb{B}_{+} (\alpha)$, satisfies $\codim \mathbb{B}_{+} (\alpha)\geq 2$ and the inequalities (\ref{eq 3}), (\ref{eq 4}) are strict, then $\gamma_1 -\gamma_2$ must be an interior point of $\Mov_1 (X)$, or equivalently, its Bott-Chern class contains a $d$-closed smooth strictly positive $(n-1, n-1)$ form.
\end{itemize}
\end{thrm}

\begin{rmk}
In the first statement, the interesting part in Theorem \ref{main result1} is that the classes $\gamma_1, \gamma_2$ are transcendental, that is, they are not given by curve classes. In the case when they are given by curve classes, by \cite{BDPP13} the numerical condition on prime divisors is already sufficient to obtain the result.
It should also be noted that, without the additional assumption on $\alpha$ in the second statement of Theorem \ref{main result1}, even if the inequalities in (\ref{eq 3}) and (\ref{eq 4}) are strict, it is still possible that the class $\gamma_1 -\gamma_2$ lies on the boundary of $\Mov_1 (X)$. From its proof below, we will see that the assumption on $\alpha$ can be weakened a little bit.
\end{rmk}

\begin{proof}[Proof of Theorem \ref{main result1}]
The proof is similar to Theorem \ref{main result}, so we only give a brief description.

For the first statement, as the proof in Theorem \ref{main result}, by applying divisorial Zariski decomposition, it is enough to show that
$$(\langle \omega^{n-1}\rangle-\langle \alpha^{n-1}\rangle)\cdot \beta \geq 0$$
for any movable class $\beta$.
Without loss of generality, we can assume that $\beta=\mu_* \widehat{\omega}$ for some modification and some K\"ahler class $\widehat{\omega}$ upstairs. By Lemma \ref{pullback}, we need to verify
\begin{equation}\label{eq pullback}
  (\langle (\mu^*\omega)^{n-1}\rangle-\langle (\mu^*\alpha)^{n-1}\rangle)\cdot \widehat{\omega} \geq 0.
\end{equation}

Applying the reverse Khovanskii-Teissier inequality for pseudo-effective classes (Lemma \ref{reverseKT}) to $\mu^*\omega, \mu^*\alpha$ yields
\begin{equation*}
  (N\cdot \langle \mu^*\omega^{n-1}\rangle)(n\mu^*\omega\cdot \langle \mu^*\alpha^{n-1}\rangle)\geq \vol(\mu^*\omega)(N\cdot \langle \mu^*\alpha^{n-1}\rangle)
\end{equation*}
for any nef class $N$. Since $\vol(\mu^*\omega)-n\mu^*\omega\cdot \langle \mu^*\alpha^{n-1}\rangle\geq 0$ and $\mu^*\omega\cdot \langle \mu^*\alpha^{n-1}\rangle>0$, we get that
\begin{equation*}
  (\langle \mu^*\omega^{n-1}\rangle-\langle \mu^*\alpha^{n-1}\rangle)\cdot N \geq 0,
\end{equation*}
which implies $(\langle (\mu^*\omega)^{n-1}\rangle-\langle (\mu^*\alpha)^{n-1}\rangle)\cdot \widehat{\omega} \geq 0$. This finishes the proof of (\ref{eq pullback}) and the first statement.

For the second statement, since $\langle \omega^n\rangle-n\omega \cdot \langle \alpha^{n-1}\rangle>0$, for $\delta>0$ small enough we have
$$\langle \omega^n\rangle-n\omega \cdot \langle (\alpha+\delta \alpha)^{n-1}\rangle>0.$$
For any non zero movable class $\beta$, applying the same argument as above to $\omega, (1+\delta)\alpha$ implies
\begin{align*}
  (\langle \omega^{n-1}\rangle-\langle \alpha^{n-1}\rangle)\cdot \beta > (\langle \omega^{n-1}\rangle-\langle (\alpha+\delta \alpha)^{n-1}\rangle)\cdot \beta \geq 0,
\end{align*}
where the first inequality follows from the assumption on $\alpha$ and Lemma \ref{movablestruct}. By similar argument in the proof of Theorem \ref{main result}, we then conclude the second statement.

This finishes the proof of Theorem \ref{main result1}.
\end{proof}

\begin{rmk}
By the above proof, we know that the only additional assumption on $\alpha$ (or $\gamma_2$) is that $\gamma_2\cdot \beta>0$ for any non zero movable $(1,1)$ class.
\end{rmk}

\section{Miscellaneous discussions}\label{secdisc}

\subsection{General restricted ``reverse Khovanskii-Teissier inequalities''} Let $X$ be a compact K\"ahler manifold of dimension $n$, and let $H, \omega, \alpha$ be any K\"ahler classes on $X$.
In the proof of Theorem \ref{main result}, we ask if we have the following inequality
\begin{equation*}
  (n-1)(H\cdot \alpha^{n-2}\cdot \omega)(H\cdot \omega^{n-2}\cdot \beta)\geq (H\cdot \omega^{n-1})(H\cdot \alpha^{n-2}\cdot \beta).
\end{equation*}
More generally, inspired by \cite[Remark 3.1]{xiao13weakmorse} and the stronger results in \cite{popovici2016sufficientbig}, we can ask the following:

\begin{itemize}
  \item Let $X$ be a compact K\"ahler manifold, and let $H, \omega, \alpha, \beta$ be any K\"ahler classes on $X$, then do we have
\begin{equation}\label{eq resKT general}
(H^{n-k-l}\cdot \alpha^{k}\cdot \omega^l)(H^{n-k-l}\cdot \omega^{k}\cdot \beta^l)\geq \frac{k!l!}{(k+l)!}(H^{n-k-l}\cdot \omega^{k+l})(H^{n-k-l}\cdot \alpha^{k}\cdot \beta^l)?
\end{equation}
\end{itemize}


\begin{rmk}
We may expect similar inequality by replacing $H^{n-k-l}$ by an arbitrary positive $(n-k-l, n-k-l)$ class. And we may also replace $\alpha^k$ (or $\beta^l$) by any (smooth) positive $(k,k)$ (or $(l,l)$) class. Indeed, the important assumption is the positivity on $\omega$: it is possible to assume that, $\omega$ is a $(k+l)$-subharmonic class; see the discussion below.
\end{rmk}

By using the method of \cite{popovici2016sufficientbig} (see e.g. \cite[Section 7]{popovici15}, \cite[Section 5.2]{lehmann2016correspondences}), it is clear that we have:

\begin{prop}\label{general reverseKT variety}
The inequality (\ref{eq resKT general}) is true whenever the class $H^{n-k-l}=\{[V]\}$, where $[V]$ is the integration current of an irreducible subvariety.
\end{prop}

We observe that the pointwise case of (\ref{eq resKT general}) is true.

\begin{prop}\label{pointwiseineq}
Let $H, \omega, \alpha, \beta$ be smooth positive $(1,1)$ forms, then we always have
\begin{equation}\label{eq resKT general pointwise}
(H^{n-k-l}\wedge \alpha^{k}\wedge \omega^l)(H^{n-k-l}\wedge \omega^{k}\wedge \beta^l)\geq \frac{k!l!}{(k+l)!}(H^{n-k-l}\wedge \omega^{k+l})(H^{n-k-l}\wedge \alpha^{k}\wedge \beta^l).
\end{equation}
\end{prop}

\begin{proof}
Following the argument of \cite[Section 5]{Dem93}, the result can be deduced directly from Proposition \ref{general reverseKT variety}.

More precisely, since (\ref{eq resKT general pointwise}) is a pointwise inequality, we just need to verify it for forms with constant coefficients.
Without loss of generality, we can assume that all the forms are strictly positive. In a suitable basis, we can assume that $H=\sqrt{-1}\sum_{j} dz^j \wedge d\bar z ^j$. Denote by $H, \omega, \alpha, \beta$ the associated cohomology classes on the abelian variety $A:=\mathbb{C}^n / \mathbb{Z}[\sqrt{-1}]^n$, then (\ref{eq resKT general pointwise}) is equivalent to the intersection number inequality (\ref{eq resKT general}) on $A$.
Since $H$ has integral periods, it is the class of a very ample divisor class (up to a constant), thus $H^{n-k-l}$ as a class on $A$ is the class (up to a constant) of an irreducible variety. Then the result follows from
Proposition \ref{general reverseKT variety}.
\end{proof}

\subsection{$(k+l)$-subharmonic class}
Inspired by \cite{popovici2016sufficientbig}, besides using complex Monge-Amp\`{e}re equations, it is natural to apply some other kind of equations to the above question. Actually, it can be generalized in the following way.

We assume $H$ is a K\"ahler metric, and assume $\omega, \alpha, \beta$ are $d$-closed $(k+l)$-subharmonic $(1,1)$ forms, that is, its eigenvalues $\lambda_1, ..., \lambda_n$ with respect to $H$ satisfy $\sigma_1 (\lambda), ..., \sigma_{k+l}(\lambda)>0$. Here, $\sigma_1, ..., \sigma_{k+l}$ are the first $(k+l)$ elementary symmetric functions. If a Bott-Chern $(1,1)$ class has a $(k+l)$-subharmonic smooth representative, then we call it a $(k+l)$-subharmonic class.

\begin{ques}\label{quesrevKT1}
Let $X$ be a compact K\"ahler manifold. Assume that $H$ is a K\"ahler class on $X$ and $\omega, \alpha, \beta$ are $(k+l)$-subharmonic class with respect to $H$, then do we have
\begin{equation}\label{eq resKT general1}
(H^{n-k-l}\cdot \alpha^{k}\cdot \omega^l)(H^{n-k-l}\cdot \omega^{k}\cdot \beta^l)\geq \frac{k!l!}{(k+l)!}(H^{n-k-l}\cdot \omega^{k+l})(H^{n-k-l}\cdot \alpha^{k}\cdot \beta^l)?
\end{equation}
\end{ques}

In the following, we use the same symbol to denote a form and its associated class.

By the result of \cite{dinew2012liouville} (see also \cite{sun2013class, szekelyhidi2015fully, zhang2015hessian} for Hermitian manifolds), we can always find a $(k+l)$-subharmonic (or ``$(k+l)$-positive'' by the terminology in \cite{szekelyhidi2015fully}) function $\phi$ satisfying
\begin{equation}\label{eqhess}
 H^{n-k-l}\wedge (\omega+i\partial \bar \partial \phi)^{k+l}= cH^{n-k-l}\wedge \alpha^{k}\wedge \beta^l,
\end{equation}
where $c=H^{n-k-l}\cdot \omega^{k+l}/H^{n-k-l}\cdot\alpha^{k}\cdot \beta^l$ is a constant.

Let $\omega_\phi := \omega+i\partial \bar \partial \phi$. Note that, since $\omega_\phi, \alpha, \beta$ are $(k+l)$-subharmonic, we have
$H^{n-k-l}\wedge \alpha^{k} \wedge \omega_\phi ^l>0$ and $H^{n-k-l}\wedge \omega_\phi ^{k}\wedge \beta^l>0$ (see e.g. \cite{blockihess}).
Denote the volume form $H^{n-k-l}\wedge \omega_\phi ^{k+l}$ by $\Phi$,  then
\begin{align*}
(H^{n-k-l}\cdot \alpha^{k}\cdot \omega^l)(H^{n-k-l}\cdot \omega^{k}\cdot \beta^l)&=\int\frac{H^{n-k-l}\wedge \alpha^{k} \wedge \omega_\phi ^l}{H^{n-k-l}\wedge \omega_\phi ^{k+l}} \Phi \int \frac{H^{n-k-l}\wedge \omega_\phi ^{k}\wedge \beta^l}{H^{n-k-l}\wedge \omega_\phi ^{k+l}}\Phi\\
&\geq \left(\int \left( \frac{H^{n-k-l}\wedge \alpha^{k} \wedge \omega_\phi ^l}{H^{n-k-l}\wedge \omega_\phi ^{k+l}}\cdot \frac{H^{n-k-l}\wedge \omega_\phi ^{k}\wedge \beta^l}{H^{n-k-l}\wedge \omega_\phi ^{k+l}}\right)^{1/2}\Phi \right)^2\\
&\geq^\dag \frac{k!l!}{(k+l)!}(H^{n-k-l}\cdot \omega^{k+l})(H^{n-k-l}\cdot \alpha^{k}\cdot \beta^l),
\end{align*}
where the last inequality ($\dag$) would follow provided a similar pointwise inequality as in Proposition \ref{pointwiseineq} holds for these $(k+l)$-subharmonic forms.

\begin{rmk}
By Proposition \ref{pointwiseineq}, if at almost every point the forms $\omega_\phi, \alpha, \beta$, considered as classes on a complex torus $A$, are K\"ahler classes when they are restricted to a general $(k+l)$-dimensional subvariety of $A$, then it is clear that we have the inequality.
However, for the general case, besides $(k+l)$-subharmonicity we are not clear if more positivity assumptions would be needed.
\end{rmk}


\bibliography{reference}
\bibliographystyle{amsalpha}

\noindent
\\
\noindent
\textsc{Department of Mathematics, Northwestern University, Evanston, IL 60208}\\
\noindent
\verb"Email: jianxiao@math.northwestern.edu"
\end{document}